\documentclass[letterpaper, 10 pt, conference]{ieeeconf}  
\usepackage{cite}
\usepackage{amsmath,amssymb,amsfonts}
\usepackage{algorithm}
\usepackage{algpseudocode}
\usepackage{graphicx}
\usepackage{textcomp}

\usepackage{graphicx}
\usepackage{multicol}
\usepackage{textcomp}
\usepackage{subcaption}
\usepackage{hyperref}

\newtheorem{theorem}{Theorem}[section]
\newtheorem{corollary}[theorem]{Corollary}
\newtheorem{lemma}[theorem]{Lemma}
\newtheorem{proposition}[theorem]{Proposition}
\newtheorem{definition}[theorem]{Definition}
\newtheorem{example}[theorem]{Example}
\newtheorem{remark}[theorem]{Remark}
\newtheorem{assumption}[theorem]{Assumption}

\usepackage{mathtools}

\usepackage{mathtools}

\usepackage{lipsum}  

\IEEEoverridecommandlockouts                              

\title{\LARGE \bf
Dynamic Programming Through the Lens of\\ Semismooth Newton-Type Methods (Extended Version)
}

\author{M. Gargiani$^{1}$, A. Zanelli$^{2}$, D. Liao-McPherson$^{1}$, T. Summers$^{3}$, J. Lygeros$^{1}$
\thanks{This
work has been supported by the European Research Council (ERC) under the OCAL project 787845. $^{1}$M. Gargiani, D. Liao-McPherson and J. Lygeros are with the Automatic Control Laboratory at ETH Zurich. $^{2}$A. Zanelli is with the Institute for Dynamic Systems and Control at ETH Zurich. $^{3}$T. Summers is with the Department of Mechanical Engineering at the University of Texas at Dallas. Correspondig author: M. Gargiani {\tt\small gmatilde@ethz.ch}. }
}

\begin{document}

\maketitle
\thispagestyle{empty}
\pagestyle{empty}

\begin{abstract}

Policy iteration and value iteration are at the core of many (approximate) dynamic programming methods. For Markov Decision Processes with finite state and action spaces, we show that they are instances of semismooth Newton-type methods to solve the Bellman equation. In particular, we prove that policy iteration is equivalent to the exact semismooth Newton method and enjoys local quadratic convergence rate. This finding is corroborated by extensive numerical evidence in the fields of control and operations research, which confirms that policy iteration generally requires few iterations to achieve convergence even in presence of a large number of admissible policies.  
We then show that value iteration is an instance of the fixed-point iteration method.
In this spirit, we develop a novel locally accelerated version of value iteration with global convergence guarantees and negligible extra computational costs. 

\end{abstract}

\section{INTRODUCTION}

\label{sec:introduction}
Approximate dynamic programming (ADP) is a powerful algorithmic strategy to handle stochastic sequential decision making problems arising in a wide range of applications, from control to games and resource allocation, to name a few.
At the core of some of the biggest success stories of ADP is an approximate version of policy iteration~\cite{alphago2016}. In particular, after an extensive offline training phase where an approximation of the optimal cost is produced, one iteration of an approximate version of policy iteration is performed (online learning). Empirical evidence suggests that this final step greatly enhances performance. 
In particular, Bertsekas in~\cite{bertsekas2022} links these success stories to the equivalence between policy iteration and Newton's method. 

The connection between policy iteration and Newton's method dates back to the late 60's~\cite{pollatschek1969}. Puterman and Brumelle~\cite{puterman1979} were among the first who exploited this connection to study the convergence properties of policy iteration for MDPs with continuous action spaces.
More recently, Santos and Ruts~\cite{santos2004} exploited this connection to analyze the asymptotic convergence of policy iteration for the discretization of a specific class of MDPs with continuous spaces. Bertsekas in~\cite{bertsekas2022} provides a graphical analysis of the connection between policy iteration and Newton's method. 
He then mathematically formalizes these visual insights by proving local quadratic convergence of policy iteration for Markov Decision Processes (MDPs) with finite state and action spaces. 
These theoretical results are corroborated by numerous computational examples which demonstrate that policy iteration achieves convergence in a remarkably small number of iterations even in presence of rounding errors and a large number of potential policies.
We refer to~\cite{bertsekas2022} for an extensive review of the related works.

In this work, we consider MDPs with finite state and action spaces and we formally show that policy iteration and value iteration are both instances of semismooth Newton-type methods. The main differences between our analysis and that of Bertsekas are that the latter only focus on policy iteration and does not deploy tools from generalized differentiation, but works in a neighborhood of the solution where the iterations can be expressed as the Newton iterations for some auxiliary continuously differentiable mapping.
We then take this connection further by developing a novel version of value iteration inspired by the fixed-point iteration method. In particular, our main contributions are the following. 
\begin{itemize}
\item In Section~\ref{sec:semismooth Newton}, we develop a unified theoretical analysis for the local convergence of semismooth Newton-type methods based on the so-called \textit{kappa condition}~\cite{diehl2016}. 
\item In Sections~\ref{sec:PI and Newton} and~\ref{sec:VI and Gradient}, we formalize mathematically the connection of policy iteration and value iteration with semismooth Newton-type methods using tools from generalized differentiation and results from Section~\ref{sec:background}. We then discuss the significant algorithmic and theoretical implications of this connection. 
\item In Section~\ref{sec: alpha VI}, we design a novel globally convergent and locally accelerated variant of value iteration with negligible additional computational cost per iteration and superior numerical performance.
\end{itemize}

\noindent\textbf{Notation.} In the following, we use $\Vert \cdot \Vert:\mathbb{R}^d\rightarrow \mathbb{R}$ to denote an arbitrary vector norm, $\Vert \cdot\Vert:\mathbb{R}^{d\times d}\rightarrow \mathbb{R}$ for its induced matrix norm, $\mathcal{B}(c, \delta)$ for the Euclidean ball with center $c\in\mathbb{R}^d$ and radius $\delta>0$, $\rho$ for the spectral radius of a matrix, $r'$ for the Jacobian operator of a differentiable function $r:\mathbb{R}^d\rightarrow \mathbb{R}^d$, $\mathbf{1}_d = \begin{bmatrix} 1 &1&\cdots&1\end{bmatrix}^\top\in\mathbb{R}^d$ and $\text{cl}\left(\mathcal{T}\right)$ and $\text{int}\left(\mathcal{T}\right)$ for the closure and the interior of a set $\mathcal{T}\subseteq \mathbb{R}^d$, respectively.

\section{BACKGROUND}
\label{sec:background}

We consider infinite horizon discounted cost problems for MDPs $\left\{\mathcal{S}, \mathcal{A}, P, g, \gamma   \right\}$ comprising a finite state space $\mathcal{S} = \left\{ 1,\dots, n  \right\}$, a finite action space $\mathcal{A} = \left\{ 1,\dots, m  \right\}$, a transition probability function $P:\mathcal{S}\times \mathcal{A} \times \mathcal{S}\rightarrow [0,1]$ that defines the probability of ending in state $s'$ when applying action $a$ in state $s$, a stage-cost function $g:\mathcal{S}\times\mathcal{A}\rightarrow \mathbb{R}$ that associates to each state-action pair a bounded cost, and a discount factor $\gamma\in(0,1)$. Throughout the paper, with a slight abuse of notation we use $\mathcal{A}(s)$ to denote the nonempty subset of actions that are allowed at state $s$, $p_{ss'}(a) = P(s, a, s')$ for the probability of transitioning to state $s'$ when the system is in state $s$ and action $a\in\mathcal{A}(s)$ is selected with $\sum_{s'\in\mathcal{S}} p_{ss'}(a) = 1$ for all $s\in\mathcal{S}$ and $a\in\mathcal{A}(s)$.

A \textit{deterministic stationary control policy} $\pi: \mathcal{S}\rightarrow \mathcal{A}$ is a function that maps states to actions, with $\pi(s)\in\mathcal{A}(s)$. We use $\Pi$ to denote the set of all deterministic stationary control policies, from now on simply called \textit{policies}. At step $t$ of the decision process under the policy $\pi\in\Pi$, the system is in some state $s_t$ and the action $a_t=\pi(s_t)$ is applied. The discounted cost $\gamma^t g(s_t,a_t)$ is accrued and the system transitions to a state $s_{t+1}$ according to the probability distribution $P(s_t, a_t, \cdot)$. This process is repeated leading to the following cumulative discounted cost
\begin{equation}\label{eq: cost of pi}
V^{\pi}(s) = \lim_{T\rightarrow \infty} \mathbb{E}\left[\, \sum_{t=0}^{T-1} \gamma^t g(s_t, \pi(s_t)) \,\,\Big|\,\,s_0=s\right],
\end{equation}
where $\left\{s_0, \pi(s_0), s_1, \pi(s_1),\dots, s_t,\pi(s_t),\dots \right\}$ is the state-action sequence generated by the MDP under policy $\pi$ with initial state $s_0$, and the expected value is taken with respect to the corresponding probability measure over the space of sequences. The transition probability distributions induced by policy $\pi$ can be compactly represented by the rows of an $n\times n$ row-stochastic matrix $\left[P^{\pi}\right]_{ss'}=p_{ss'}(\pi(s))$ for all $s,s' \in \mathcal{S}$ and the costs induced by policy $\pi$ by the vector $g^{\pi} = \begin{bmatrix}
g(1, \pi(1)) &\cdots & g(n, \pi(n)) 
\end{bmatrix}^\top\in\mathbb{R}^{n}$.
The optimal cost is defined as 
\begin{equation}\label{eq: optimal cost}
V^*(s) \vcentcolon= \min_{\pi\in\Pi} V^{\pi}(s)\quad\forall s \in \mathcal{S}. 
\end{equation}
Any policy $\pi^*\in\Pi$ that attains the optimal cost is called an optimal policy. 
Notice that in~\eqref{eq: optimal cost} we restrict our attention to stationary deterministic policies as in our setting there exists a policy in this class that attains $V^*$~\cite{bertsekas2012DPOC}.
The optimal cost admits a \textit{recursive} definition known as the \textit{Bellman equation}
\begin{equation}\label{eq: Bellman equation}
V^*(s) = \min_{a\in \mathcal{A}(s)} \left\{ g(s,a) + \gamma \sum_{s'\in\mathcal{S}} p_{ss'}(a) V^*(s') \right\}\,\, \forall s\in\mathcal{S}.
\end{equation}
Equation~\eqref{eq: cost of pi} admits an analogous recursive definition known as the Bellman equation associated with policy $\pi$. In the considered setting, the cost function associated with policy $\pi$ and the optimal cost function can be represented by $V^{\pi}\in\mathbb{R}^n$ and $V^*\in\mathbb{R}^n$, where the $s$-th element is given by~\eqref{eq: cost of pi} and~\eqref{eq: optimal cost} evaluated at $s$, respectively.

\subsection{Dynamic Programming}
\label{sec:DP}
Dynamic Programming (DP) comprises the methods for solving stochastic optimal control problems by solving the Bellman equation~\cite{bertsekas2012DPOC}. 
Here we are interested in DP algorithms in the classes of value iteration (VI) and policy iteration (PI). 
Starting from Equation~\eqref{eq: Bellman equation}, we define a nonsmooth mapping $T:\mathbb{R}^n \rightarrow \mathbb{R}^n$, known as the \textit{Bellman operator}, by
\begin{equation*}
(TV)(s) = \min_{a\in\mathcal{A}(s)}\left\{ g(s,a) + \gamma \sum_{s'\in\mathcal{S}} p_{ss'}(a)V(s') \right\}\,\,\forall s\in\mathcal{S}.
\end{equation*}
An analogous linear operator $T^{\pi}:\mathbb{R}^n \rightarrow \mathbb{R}^n$ can be defined for the Bellman equation associated with policy $\pi$ as
\begin{equation*}
(T^{\pi}V)(s) = g(s, \pi(s)) + \gamma \sum_{s'\in\mathcal{S}} p_{ss'}(\pi(s))V(s') \quad\forall s\in\mathcal{S}.
\end{equation*}
Given the cost vector $V$, any policy $\pi$ such that  
\begin{equation}\label{eq: greedy policy}
\pi(s) \in \arg\!\!\min_{a\in\mathcal{A}(s)} \left\{ g(s,a) + \gamma \sum_{s'\in\mathcal{S}} p_{ss'}(a)V(s') \right\} \,\,\forall s\in\mathcal{S}
\end{equation}
is called \textit{greedy} with respect to the cost $V$.
It can be shown that the Bellman operator is contractive~\cite{bertsekas2012DPOC} and, thanks to the Banach Theorem~\cite{rockafellar76}, admits a unique fixed point $V^*$. Moreover, the corresponding Picard-Banach iteration converges asymptotically to the fixed point from any initial value $V$, i.e.
\begin{equation}\label{eq: value iteration}
\lim_{k\rightarrow \infty} T^k V = V^*.
\end{equation}
This is at the core of VI, which repeatedly applies the $T$ operator starting from an arbitrary finite cost.
The generated sequence linearly converges to $V^*$ with a $\gamma$-contraction rate.   

An alternative method to solve Equation~\eqref{eq: Bellman equation} is PI (Algorithm~\ref{alg: PI}). With PI, we start from an arbitrary initial policy and alternate \textit{policy evaluation} (step 3) and \textit{policy improvement} (step 4) until convergence. The policy evaluation step at iteration $k$ computes the cost $V^{\pi_k}$ associated with the current policy $\pi_k$. This requires the solution of a system with $n$ linear equations, which is generally computationally demanding for MDPs with large state spaces. The policy is then updated by extracting a greedy policy associated with $V^{\pi_k}$ in the policy improvement step. Unlike VI, PI converges in a finite number of iterations since the policy, and therefore also its cost, are improved at each iteration and since, by the finiteness of $\mathcal{S}$ and $\mathcal{A}$, there only exists a finite number of policies. It is nonetheless important to characterize its convergence rate and asymptotic behavior since, for large state and action spaces, the number of iterations could be prohibitive (exponential in $n$ and $m$). By exploiting the properties of the Bellman operator, we can show that PI is globally $\gamma$-contractive, which is similar to VI. Extensive empirical evidence, however, suggests that PI has superior convergence properties and generally requires considerably fewer iterations than VI. From a computational viewpoint, the per-iteration costs of PI with direct inversion amount to $\mathcal{O}( n ^3 + m\cdot n^2)$ versus the $\mathcal{O}(m\cdot n^2)$ of VI. 
\begin{algorithm}[H]
\caption{Exact Policy Iteration}
\label{alg: PI}
\begin{algorithmic}[1]
    \State \textbf{Initialization:} select an arbitrary initial policy $\pi_0$ and set $k=0$
    \While{cost has not converged}
        \State $V^{\pi_k} = (I-\gamma P^{\pi_k})^{-1}g^{\pi_k}$
        
        \State $\pi_{k+1} = \tilde{\pi}$ with $\tilde{\pi}\in\text{GreedyPolicy}(V^{\pi_k})$ according to~\eqref{eq: greedy policy}
        \State $k \leftarrow k+1$
    \EndWhile
\end{algorithmic}
\end{algorithm}

\subsection{Generalized Differentiation \& Semismooth Newton-Type Methods}
\label{sec:semismooth Newton}
Consider the following nonlinear root finding problem
\begin{equation}\label{eq: root finding}
r(\theta) = 0\,,
\end{equation}
where $r:\mathbb{R}^d \rightarrow \mathbb{R}^d$ is a locally Lipschitz-continuous vector-valued function. 
A vector $\theta^*\in\mathbb{R}^d$ that verifies~\eqref{eq: root finding} is called \textit{root} or \textit{solution} of the nonlinear equation~\eqref{eq: root finding}. In general, we can not rely on smooth optimization methods~\cite{izmasolo14} to solve~\eqref{eq: root finding} since $r$ can be nonsmooth, so its Jacobian $r^\prime(\theta)\in\mathbb{R}^{d\times d}$ might not exist. We therefore need to introduce some notions of generalized differentiability from nonsmooth analysis~\cite{clarke1990optimization}, such as the \textit{B-differential} and Clarke's \textit{generalized Jacobian}.
Since $r$ is a locally Lipschitz-continuous map, the Rademacher Theorem~\cite{rademacher1919partielle} implies that it is differentiable almost everywhere and we denote with $\mathcal{M}_r$ the set of all points where $r$ is differentiable. Another fundamental implication of the Rademacher Theorem is the definition of the B-differential of $r$ at $\theta\in\mathbb{R}^d$ as the set 
\begin{equation*}
\partial_B r(\theta) \!=\! \left\{\!J\! \in \mathbb{R}^{d\times d}\big\vert\, \exists\left\{ \theta_k\right\}\!\subset\! \mathcal{M}_r\!:\!\left\{ \theta_k\right\}\!\rightarrow\! \theta ,\left\{ r^\prime(\theta)\right\}\!\rightarrow\! J \right\}.
\end{equation*}
We denote with $\partial r(\theta)$ Clarke's generalized Jacobian of $r$ at $\theta\in\mathbb{R}^d$, which is defined as the convex hull of $\partial_B r(\theta)$. Consequently, $\partial_B r(\theta) \subseteq \partial r(\theta)$. These sets are always nonempty when evaluated at points where the function is Lipschitz continuous~\cite[Proposition 1.51]{izmasolo14}.
If $r$ is continuously differentiable at $\theta$, then $\partial r(\theta) = \partial_B r(\theta) = \left\{ r^{\prime}(\theta) \right\}$. Otherwise, $\partial_B r(\theta)$ and, consequently, $\partial r(\theta)$ are not necessarily singletons.

The B-differential and Clarke's generalized Jacobian are of practical interest only if we can compute at least some of their elements. Because of the lack of sharp calculus rules, this can be done only in few cases, depending on the structure of $r$. For instance, consider the class of piecewise continuously differentiable functions on $\mathbb{R}^d$~\cite{10.1145/2491491.2491493}, which is formally characterized by the following definition. 
\begin{definition}[PC$^1$ functions]
Let $f:\mathbb{R}^d\rightarrow \mathbb{R}^o$ be a continuous vector-valued function and $n_p$ be some positive integer. The function $f$ is said to be \textit{piecewise continuously differentiable} of order $1$ (PC$^1$) if there exist finitely many continuously differentiable functions $\left\{ f_i \right\}_{i=1}^{n_p}$ on $\mathbb{R}^d$, called \textit{selection functions}, such that $f(\theta)\in \left\{ f_i(\theta) \right\}_{i=1}^{n_p}$ for all $\theta\in\mathbb{R}^d$. In addition, $f_i$ is \textit{active} at $\bar{\theta}\in\mathbb{R}^n$ if $f(\bar{\theta}) = f_i(\bar{\theta})$ and \textit{essentially active} if $\bar{\theta}\in \text{cl}( \text{int}(  \{  \theta\in\mathbb{R}^d \,:\, f(\theta) = f_i(\theta) \} ) )$. 
\end{definition}

We denote with $\mathcal{F}_{f}(\bar{\theta})$ the collection of essentially active functions at $\bar{\theta}$.
Piecewise affine functions are an example of PC$^{1}$ functions with affine selection functions and are particularly relevant in the context of DP as it will be discussed in Section~\ref{sec:newton-type DP}. 

The following proposition (Lemma 2.10 in~\cite{10.1145/2491491.2491493}) gives a representation of the B-differential for PC$^1$ functions. This representation can be used to determine a $J\in\partial_B f(\theta)$ in cases where we can compute the Jacobian matrix of at least one of the essentially active selection functions at $\theta\in\mathbb{R}^d$.
\begin{proposition}\label{proposition B-differential of PC functions}
Let $f:\mathbb{R}^d\rightarrow \mathbb{R}^o$ be a PC$^1$ function. The B-differential of $f$ at $\theta\in\mathbb{R}^d$ is
$\partial_B f(\theta) = \left\{ f_i^{\prime}(\theta)\,:\,f_i\in\mathcal{F}_f(\theta) \right\} \,.$
\end{proposition}
\begin{example}\label{example}
Consider the following piecewise affine function: $f(\theta) = 2 \theta - 5$ if $\theta>5$, $f(\theta) = \theta$ if $\theta = 5$ and $f(\theta) = -2 \theta + 15$ if $\theta<5$.
Then $\partial_B f(5) = \left\{2, -2\right\}$ since $\text{int}(\{ \theta\in \mathbb{R}\,:\,f(\theta) = \theta \}) = \emptyset$ and $\partial_B f(\theta) = f'(\theta)$ for all $\theta\in\mathbb{R}\setminus \left\{5\right\}$.
\end{example}

We refer to~\cite{10.1145/2491491.2491493} for more details on the computation of elements in Clarke's generalized Jacobian for piecewise continuous functions and to Chapter 1 in~\cite{izmasolo14} for functions with different structures.

The Newton method~\cite{izmasolo14} is not directly applicable to solve~\eqref{eq: root finding} because of the nonsmoothness. The extension of the Newton method to nonsmooth equations dates back to at least~\cite{kummer88} and is generally known as the semismooth Newton method~\cite{qi1993nonsmooth},~\cite{izmasolo14}. Similarly to the Newton method, instead of solving directly~\eqref{eq: root finding}, the semismooth Newton method solves a series of linear equations that locally approximate~\eqref{eq: root finding}, but the Jacobian matrix in the Newtonian iteration system is replaced by an element from Clarke's generalized Jacobian. In particular, the semismooth Newton method generates a sequence of iterates $\left\{ \theta_k \right\}$ where $\theta_0\in\mathbb{R}^d$ is the initial approximation of the root and, for any $k\geq 0$, $\theta_{k+1}$ is computed as a solution of the linear equation
$
r(\theta_k) + J_k \left( \theta_{k+1} - \theta_k \right)= 0\,,
$
with $J_k\in\partial r(\theta_k)$. When $J_k$ is nonsingular, then the iterate $\theta_{k+1}$ can be computed in closed-form as follows
\begin{equation}\label{eq: semismooth Newton iteration}
\theta_{k+1} = \theta_k - J_k^{-1}r(\theta_k)\,.
\end{equation}
Under certain assumptions, the semismooth Newton method enjoys fast local quadratic convergence, but the cost per iteration with direct inversion is in the order of $\mathcal{O}(d^3)$. In addition, as discussed, it may be difficult to obtain an element from Clarke's generalized Jacobian. 
These are some of the main motivations behind the design of different variants of the semismooth Newton method of the form
\begin{equation}\label{eq: Newton-type iteration}
r(\theta_k) + B_k(\theta_{k+1} - \theta_k) = 0\,,
\end{equation}  
where $B_k\in\mathbb{R}^{d \times d}$. These variants, collectively known as semismooth Newton-type methods~\cite{izmasolo14}, can lead to lower computational costs while maintaining acceptable convergence rates.
Clearly, if $B_k \in \partial r(\theta_k)$, then we recover the semismooth Newton method. Among the most frequently used semismooth Newton-type methods, we recall the \textit{fixed-point iteration method}, where $B_k = \alpha_k I$ with $\alpha_k\neq 0$~\cite{facchineipangvol22003}. 

Before proceeding with the formal characterization of the local convergence rate of semismooth Newton-type methods, we need to introduce the notions of strong semismoothness~\cite[Subsection 1.4.2]{izmasolo14} and CD-regularity~\cite[Remark 1.65]{izmasolo14}.
\begin{definition}[strong semismoothness]
A function $f:\mathbb{R}^d \rightarrow \mathbb{R}^o$ is strongly semismooth at $\theta\in\mathbb{R}^d$ if it is locally Lipschitz-continuous at $\theta$, directionally differentiable at $\theta$ in every direction, and the following estimate holds as $\xi \in\mathbb{R}^d$ tends to zero
\vspace{-0.3cm}
$$ \sup_{J\in\partial f(\theta + \xi)}\Vert f(\theta + \xi) - f(\theta) - J\xi \Vert = \mathcal{O}(  \Vert \xi \Vert^2 )\,. $$
\end{definition}

\begin{definition}[CD/BD-regularity]
A function $f:\mathbb{R}^d\rightarrow \mathbb{R}^o$ is CD-regular (BD-regular) at $\theta\in\mathbb{R}^d$ if each matrix $J\in\partial f(\theta)$ ($J\in\partial_B f(\theta)$) is nonsingular.
\end{definition}
The function in Example~\ref{example} is strongly semismooth and BD-regular everywhere, but not CD-regular at $\theta=5$, since $0\in\partial f(5)$.

The following theorem characterizes the local contraction of a semismooth Newton-type sequence generated by Algorithm~\ref{alg: semismooth Newton-type}. Similar \textit{a-posteriori} results based on perturbation analysis can be found in~\cite{izmasolo14}.

\begin{theorem}\label{theorem local contraction}
Let $r:\mathbb{R}^d\rightarrow \mathbb{R}^d$ be strongly semismooth at $\theta^*\in\mathbb{R}^d$, $L>0$ and $\kappa \in [0,1)$ a constant. Then the following statements hold.
\begin{enumerate}
\item For any nonsingular matrix $B\in\mathbb{R}^{d\times d}$ such that $\Vert B^{-1} \Vert \leq L$ and $\exists\,\, J\in\partial r(\theta)$ for which $\Vert B^{-1}\left( B - J\right) \Vert \leq \kappa$, then 
\begin{equation}\label{eq: local contraction first assertion}
\Vert \theta - B^{-1}r(\theta) - \theta^*  \Vert \leq \kappa \,\Vert \theta - \theta^* \Vert + \mathcal{O}( \Vert  \theta - \theta^* \Vert^2)\,.
\end{equation}
 
\item There exist an open neighborhood of $\theta^*$ such that, for any $\theta_0$ in the neighborhood and any sequence of nonsingular matrices $\left\{ B_k\right\}\subseteq \mathbb{R}^{d\times d}$ such that, for all $k$, $\Vert B_k^{-1} \Vert\leq L$ and $\exists\,\, J_k\in\partial r(\theta_k)$ for which the kappa condition
\begin{equation}\label{eq: kappa_condition}
\Vert B_k^{-1}\left( B_k - J_k \right) \Vert \leq \kappa_k \leq \kappa 
\end{equation}

is verified, the sequence $\left\{\theta_k\right\}\subseteq \mathbb{R}^d$ generated by Algorithm~\ref{alg: semismooth Newton-type} converges to $\theta^*$ and
\begin{equation}\label{eq: local contraction}
\Vert \theta_{k+1} - \theta^*  \Vert \leq \kappa_k \,\Vert \theta_k - \theta^* \Vert + \mathcal{O}( \Vert  \theta_k - \theta^* \Vert^2)\,.
\end{equation}
\end{enumerate} 
 
\end{theorem}
\vspace{0.1cm}
\begin{proof}
We start by proving the first assertion. Since $r(\theta^*)=0$,
\begin{equation*}
\begin{aligned}
    \theta \!-\! B^{-1}r(\theta) - \theta^* 
    &\!= B^{-1}\!\left(B (\theta - \theta^*) \right)- B^{-1}\!\left( r(\theta) - r(\theta^*) \right).
\end{aligned}
\end{equation*}
We now add and subtract the term $B^{-1}J(\theta - \theta^*)$, where $J\in\partial r(\theta)$ such that $\Vert B^{-1}\left( B - J \right) \Vert \leq \kappa$ 
\begin{equation}\label{eq: eq1 theorem}
\begin{aligned}
    \theta - B^{-1}r(\theta)& - \theta^* 
     = B^{-1}(B - J)(\theta - \theta^*) \\
    &\quad- B^{-1}(r(\theta) - r(\theta^*) - J(\theta - \theta^*))\,.
    \end{aligned}
\end{equation}
By taking the norm on both sides of Equation~\eqref{eq: eq1 theorem}, we obtain
\begin{gather}\label{eq: eq1 theorem semismooth newton-type}
\begin{aligned}
    &\Vert \theta - B^{-1}r(\theta) - \theta^* \Vert \\
    &= \Vert
     B^{-1}(B -\! J)(\theta - \theta^*) \!-\! B^{-1}(r(\theta) - r(\theta^*)\!-\! J(\theta - \theta^*))\Vert \\
      & \overset{(a)}{\leq} \!\Vert
     B^{-1}(B - J)(\theta - \theta^*) \Vert \\
     &\quad\quad\!+ \Vert B^{-1}(r(\theta) - r(\theta^*) - J(\theta - \theta^*))\Vert  \\
     &\overset{(b)}{\leq} \!\Vert
     B^{-1}(B - J)\Vert \Vert \theta - \theta^*\Vert\\
     &\quad\quad\!+ \Vert B^{-1}\Vert \Vert r(\theta) - r(\theta^*) - J(\theta - \theta^*)\Vert  \\
     &\leq \Vert
     B^{-1}(B - J)\Vert \Vert \theta - \theta^*\Vert\\
     &\quad\quad+ L \Vert (r(\theta) - r(\theta^*) - J(\theta - \theta^*))\Vert \\
     &\overset{(c)}{=}\Vert
     B^{-1}(B - J)\Vert \Vert \theta - \theta^*\Vert  +  \mathcal{O} (\Vert \theta - \theta^*\Vert^2)\,,
\end{aligned}\raisetag{4.5\baselineskip}
\end{gather}
where $(a)$ follows from the triangle inequality, $(b)$ from the sub-multiplicativity of the norm and $(c)$ from the strong semismoothness of $r$.
The final result follows from that fact that $\Vert B^{-1}\left(B - J \right) \Vert \leq \kappa$. 
For $\theta_k\in \mathbb{R}^d$, Equation~\eqref{eq: Newton-type iteration} has a unique solution $\theta_{k+1}$ given by~\eqref{eq: Newton-type iteration closed-form}. In addition, from~\eqref{eq: local contraction first assertion} it follows that for any $q\in (\kappa,1)$, there exists $\delta>0$ such that the inclusion $\theta_k\in \mathcal{B}(\theta^*, \delta )$ implies that $\Vert \theta_{k+1} - \theta^* \Vert\leq q \Vert \theta_k - \theta^*  \Vert$ and therefore $\theta_{k+1} \in \mathcal{B}(\theta^*, \delta)$.
It follows that any starting point $\theta_0\in \mathcal{B}(\theta^*, \delta)$ uniquely deﬁnes a speciﬁc sequence of iterates $\left\{\theta_k \right\}$ of Algorithm~\ref{alg: semismooth Newton-type}; this sequence is contained in $\mathcal{B}(\theta^*, \delta)$ and converges to $\theta^*$. Finally, starting from~\eqref{eq: eq1 theorem semismooth newton-type} and by exploiting~\eqref{eq: Newton-type iteration closed-form} and the kappa condition, we obtain~\eqref{eq: local contraction}.~~\QED
\end{proof}
Theorem~\ref{theorem local contraction} shows that the local convergence rate of semismooth Newton-type methods strongly depends on the choice of $\left\{ B_k \right\}$. In particular, we obtain quadratic convergence if $\kappa = 0$, superlinear convergence if $\kappa_k\rightarrow 0$ as $k\rightarrow \infty$ and linear convergence if $\kappa_k = \kappa$ for all $k$ with $\kappa \in (0, 1)$. 

The following corollary characterizes the local convergence of the exact semismooth Newton method (see also Theorem 2.42 in~\cite{izmasolo14}).

\begin{corollary}\label{corollary Newton local contraction}
Let $r$ be strongly semismooth and CD-regular at $\theta^*$. Provided that $\theta_0$ is close enough to $\theta^*$, the sequence $\left\{ \theta_k\right\}$ generated by the semismooth Newton method iteration~\eqref{eq: semismooth Newton iteration} with starting point $\theta_0$ converges to $\theta^*$ according to
\vspace{-0.2cm}
\begin{equation*}
\Vert \theta_{k+1} - \theta^* \Vert = \mathcal{O}(\Vert \theta_k - \theta^* \Vert^2 )\,.
\end{equation*}
\end{corollary}
\vspace{0.1cm}
\begin{proof}
From Proposition 1.51 and Lemma A.6 in~\cite{izmasolo14} it follows that there exists a neighborhood $U$ of $\theta^*$ and a finite constant $L>0$ such that $J$ is nonsingular and $\Vert J^{-1} \Vert\leq L$ for all $J \in \partial r(\theta)$ and for all $\theta\in U$. The final result follows from Theorem~\ref{theorem local contraction} by setting $B_k = J_k$ and considering $\theta_0\in\mathcal{B}(\theta^*, \delta)$ with $\delta$ sufficiently small such that $\mathcal{B}(\theta^*, \delta)\subset U$.~~\QED
\end{proof}
\begin{remark}\label{remark CD and BD regularity}
If at each iteration of the semismooth Newton method we select $J_k$ from $\partial_B r(\theta_k)$, then the CD-regularity assumption can be replaced by the weaker assumption of BD-regularity of $r$ at $\theta^*$. The proof is analogous but instead of considering $J_k \in \partial r(\theta_k)$ we consider $J_k \in \partial_B r(\theta_k)$. See~\cite[Remark 2.54]{izmasolo14} for a more detailed discussion.
\end{remark}
\begin{algorithm}[H]
\caption{Semismooth Newton-Type Method}
\label{alg: semismooth Newton-type}
\begin{algorithmic}[1]
    \State \textbf{Initialization:} select $\theta_0\in\mathbb{R}^d$, $tol\geq 0$ and set $k=0$
    \While{$\Vert r(\theta_k) \Vert>tol$}
        \State select $B_k\in\mathbb{R}^{d\times d}$ nonsingular and compute 
        \begin{equation}\label{eq: Newton-type iteration closed-form}
        \theta_{k+1} = \theta_{k} - B_k^{-1}r(\theta_k)
        \end{equation}
        
        \State $k \leftarrow k+1$
    \EndWhile
\end{algorithmic}
\end{algorithm}

\section{SEMISMOOTH NEWTON-TYPE DYNAMIC PROGRAMMING}
\label{sec:newton-type DP}
In this section we formalize the connection of PI and VI with semismooth Newton-type methods. Such a connection has far-reaching consequences. By adopting this different perspective on DP methods, we can indeed deploy the well-established semismooth Newton-type theory to analyze existing DP methods and design novel ones, with favorable local contraction rates and efficient iterations. 

We start by looking at the Bellman equation~\eqref{eq: Bellman equation} as a nonlinear root finding problem, where $r(\theta) = \theta - T\theta$, $r:\mathbb{R}^{n}\rightarrow \mathbb{R}^n$ and the $s$-th component is \begin{equation*}
\theta_s - \!\!\min_{a\in\mathcal{A}(s)} \!\left\{g(s,a) + \gamma \sum_{s'=1}^n p_{ss'}(a) \theta_j
\right\}\,.
\end{equation*}
We call $r$ the \textit{Bellman residual function}.
\begin{figure}[h]
\centering
\includegraphics[width=0.45\textwidth]{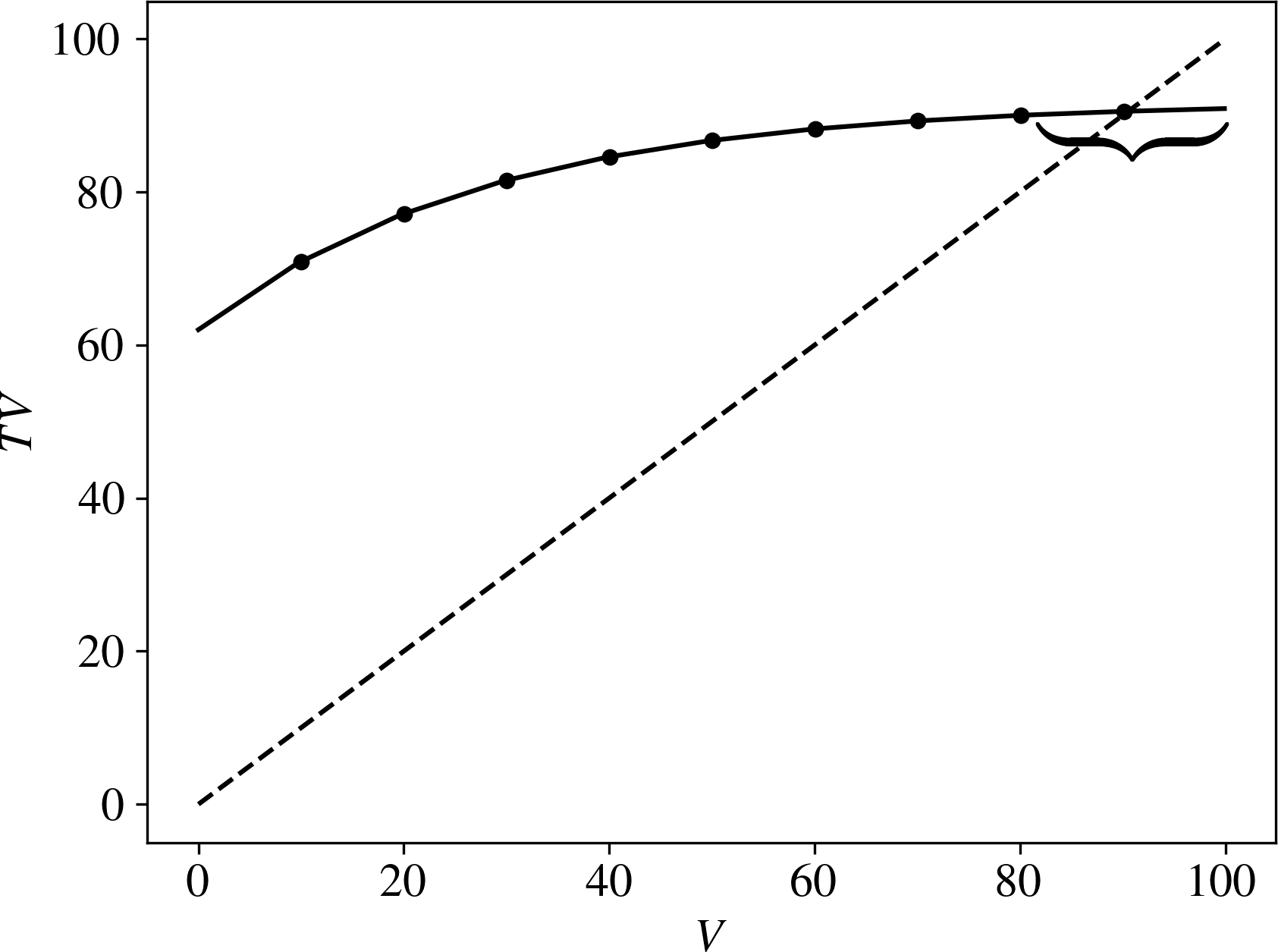}
\caption{Visualization of the Bellman operator $T$ and corresponding Bellman equation for a 1-dimensional case. The optimal cost $V^*$ corresponds to the intersection point of the graph of $TV$ with the 45 degree line. Bertsekas in~\cite{bertsekas2022} proves local quadratic convergence of PI for a region that is included in the segments within the curly brackets.}
\label{figure: pieces_BE}
\end{figure}
\begin{figure}[h]
\centering
\includegraphics[width=0.45\textwidth]{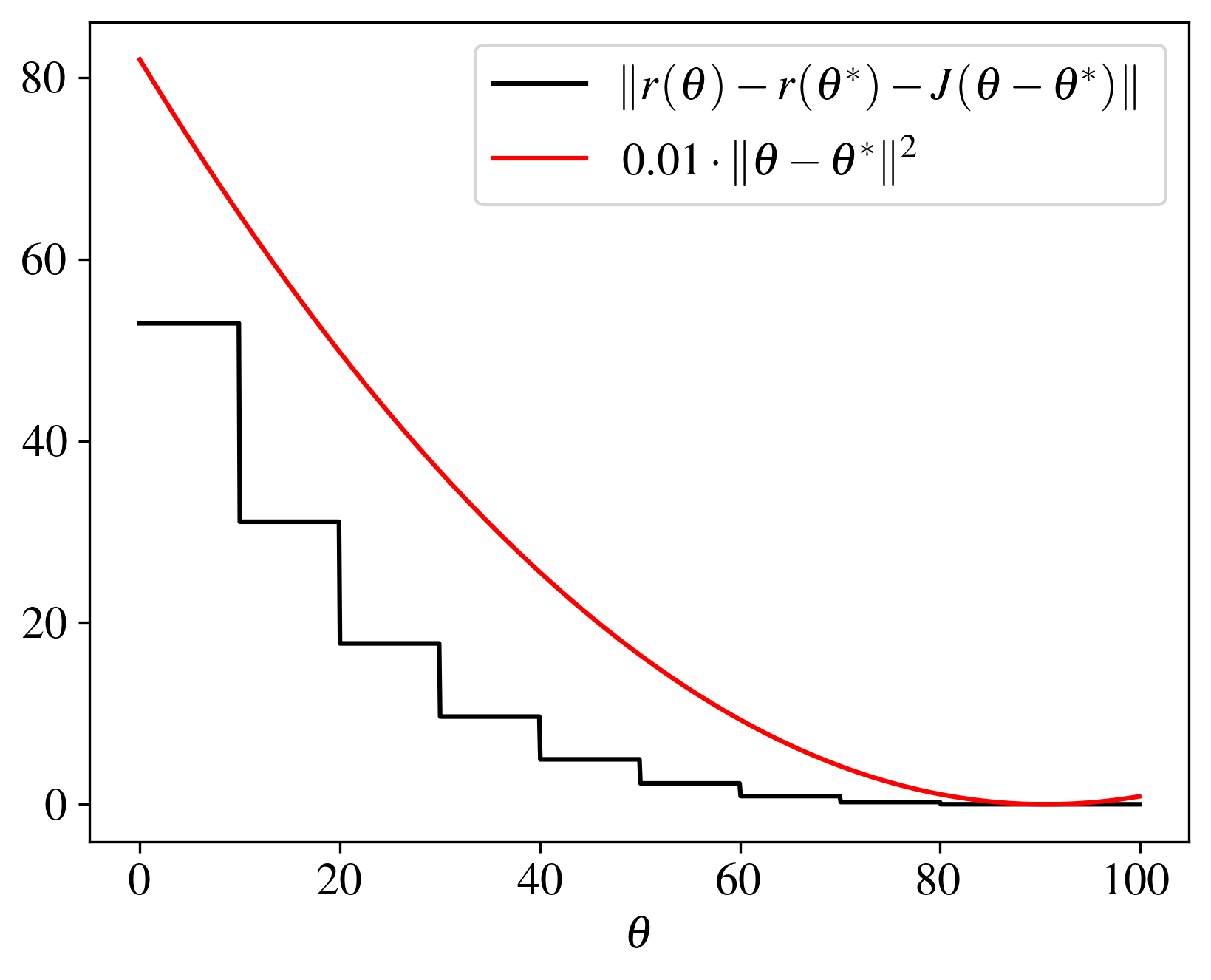}
\caption{Visualization of the region of attraction for the case of Figure~\ref{figure: pieces_BE} under our analysis. As we can clearly see from this graphical representation, the region of attraction from our analysis can be effectively larger than the one considered in~\cite{bertsekas2022}. }
\label{figure: region_attraction}
\end{figure}
Clearly, every component is piecewise affine and therefore convex~\cite{boyd2004convex}, because it is the sum of the identity map with the negative minimum of a finite collection of affine functions, one per admissible action. Consequently, the Bellman residual function is convex and continuous. Looking at the set of the admissible policies and based on the relation between $T$ and $T^{\pi}$, we can rewrite the Bellman residual function as follows
\begin{equation}\label{eq: Bellman residual function}
r(\theta) = \theta - \min_{\pi\in\Pi}\left\{ T^{\pi}\theta\right\} = \theta - \min_{\pi\in\Pi} \!\left\{g^{\pi} + \gamma P^{\pi} \theta
\right\}\,,
\end{equation}
where $T^{\pi}\theta = g^{\pi} + \gamma P^{\pi} \theta$ is an affine function of $\theta$.
Consequently, the Bellman residual function is piecewise affine since it is continuous and there exist $\vert \Pi \vert$ affine selection functions $\left\{\theta - T^{\pi}\theta\right\}_{\pi\in\Pi}$ such that $r(\theta)\in \left\{\theta - T^{\pi}\theta\right\}_{\pi\in\Pi}$ for all $\theta\in\mathbb{R}^n$.
Because of its piecewise affine structure, the Bellman residual function is globally Lipschitz continuous (Proposition 4.2.2 in~\cite{facchineipangvol12003}) and strongly semismooth everywhere (Proposition 7.4.7 in~\cite{facchineipangvol22003}). 

The following lemma characterizes the relation between greedy policies and active selection functions at $\theta\in\mathbb{R}^n$.
\begin{lemma}
Let $\tilde{\Pi}_{\theta}\subseteq \Pi$ denote the set of the greedy policies with respect to the cost-vector $\theta\in\mathbb{R}^n$. Then $r(\theta) = \theta - T^{\pi}\theta$ for all $\pi\in\tilde{\Pi}_{\theta}$. In other terms, $\left\{ \theta - T^{\pi}\theta \right\}_{\pi\in\tilde{\Pi}_{\theta}}$ is the collection of the active selection functions of $r$ at $\theta$.
\end{lemma}
\begin{proof}
The proof follows directly from the definition of greedy policy~\eqref{eq: greedy policy}. In particular, a policy $\pi$ is greedy with respect to the cost-vector $\theta\in\mathbb{R}^n$ if  $T^{\pi}\theta = T\theta$.~~\QED
\end{proof}

The next definition introduces the concept of \textit{spurious greedy policy}, which will later be used together with Proposition~\ref{proposition B-differential of PC functions} to characterize the B-differential of the Bellman residual function.
\begin{definition}[spurious greedy policy]
Let $\bar{\theta}\in\mathbb{R}^n$. $\pi\in\tilde{\Pi}_{\bar{\theta}}$ is a spurious greedy policy for the cost-vector $\bar{\theta}$ if $\text{int}(\{ \theta\in\mathbb{R}^n \,:\, r(\theta) = \theta - T^{\pi}\theta  \} ) = \emptyset\,.$
\end{definition}

In other terms, a greedy policy $\pi\in\tilde{\Pi}_{\theta}$ is spurious if there exist $s\in\mathcal{S}$  for which for all $\epsilon>0$, $\pi(s)$ is not greedy with respect to any $\tilde{\theta}_s\neq \theta_s$ with $\vert \theta_s - \tilde{\theta}_s \vert \leq \epsilon$. We denote with $\tilde{\Pi}^S_{\theta}$ the subset of $\tilde{\Pi}_{\theta}$ comprising the spurious greedy policies.

The next proposition characterizes the B-differential of the Bellman residual function.
\begin{proposition}\label{proposition of B-differential of Bellman residuals}
Let $r:\mathbb{R}^n \rightarrow \mathbb{R}^n$ be the Bellman residual function. The B-differential of $r$ at $\theta\in\mathbb{R}^n$ is the set
\begin{equation}\label{eq: B-differential of Bellman residuals}
\partial_B r(\theta) = \left\{ I - \gamma P^{\pi} \,|\, \forall \pi \in \tilde{\Pi}_{\theta}\setminus \tilde{\Pi}^{S}_{\theta} \right\}\,.
\end{equation}
In addition, $r$ is globally CD-regular.
\end{proposition} 
\begin{proof}
From the definition of essentially active selection functions and spurious greedy policies, it follows that $\mathcal{F}_r(\theta) = \left\{ \theta - T^{\pi}\theta\,|\,\forall\pi\in\tilde{\Pi}_{\theta}\setminus\tilde{\Pi}^S_{\theta} \right\}$. From Proposition~\ref{proposition B-differential of PC functions} and since $\left(\theta - T^{\pi}\theta \right)' = I - \gamma P^{\pi}$ for any $\pi\in\Pi$, we conclude that the B-differential of $r$ is given by the set in~\eqref{eq: B-differential of Bellman residuals}. Since $P^{\pi}$ is a row-stochastic matrix, its eigenvalues lie within the unit circle of the complex plane. Thus $I-\gamma P^{\pi}$ with $\gamma \in (0,1)$ has no eigenvalue equal to zero. We can therefore conclude that all the matrices in the B-differential of $r$ are nonsingular and therefore $r$ is BD-regular. Finally, since the convex combination of row stochastic matrices is a row stochastic matrix, we can conclude that $r$ is CD-regular.~~\QED 
\end{proof}
\subsection{Policy Iteration}
\label{sec:PI and Newton}
We start by introducing an assumption on the sets of the spurious greedy policies, which excludes the presence of selection functions that are active but not essentially active.
\begin{assumption}\label{assumption on spurious greedy policies}
We assume that
$\tilde{\Pi}_{\theta}^S = \emptyset$ for all $\theta \in \mathbb{R}^n$.
\end{assumption}
The following proposition characterizes the connection between PI and the semismooth Newton method.
\begin{proposition}\label{proposition PI-Newton}
Under Assumption~\ref{assumption on spurious greedy policies}, PI is an instance of the semismooth Newton method to solve the Bellman residual function~\eqref{eq: Bellman residual function}. Hence, the local contraction is quadratic.
\end{proposition}
\begin{proof}
Let $\left\{\theta^{\text{PI}}_{k}\right\}$ denote the iterates of Algorithm~\ref{alg: PI}. We show by induction that, through an appropriate choice of $J_k$, we can generate iterates $\left\{ \theta^{\text{N}}_k \right\}$ of the semismooth Newton method for the Bellman residual function such that $\theta^{\text{PI}}_k \!=\! \theta^{\text{N}}_k$ for all $k$. Assume that $\theta^{\text{PI}}_{k} = \theta^{\text{N}}_{k} = \theta_{k}$ and let $\pi_{k+1}\in \,\tilde{\Pi}_{\theta_k}$ be the greedy policy selected by PI at the $k$-th policy improvement step. Then, from Algorithm~\ref{alg: PI}, it follows that
$
\theta_{k+1}^{\text{PI}} = (I-\gamma P^{\pi_{k+1}})^{-1}g^{\pi_{k+1}}\,.
$
From Assumption~\ref{assumption on spurious greedy policies} and Proposition~\ref{proposition of B-differential of Bellman residuals}, we have that $I-\gamma P^{\pi_{k+1}}$ is invertible and belongs to $\partial_B r(\theta_k)$. Recall in addition that, from the definition of greedy policy, $T^{\pi_{k+1}}\theta_k = T\theta_k$. Therefore, the $(k+1)$-th semismooth Newton iterate with $J_k = I-\gamma P^{\pi_{k+1}}$ is
\begin{equation*}
\begin{aligned}
\theta^{\text{N}}_{k+1} &= \theta_k - (I-\gamma P^{\pi_{k+1}})^{-1} r(\theta_k)\\
&= \theta_k - (I-\gamma P^{\pi_{k+1}}) \left( \theta_k - g^{\pi_{k+1}} - \gamma P^{\pi_{k+1}}\theta_k\right)\\
&= \theta_k - (I-\gamma P^{\pi_{k+1}})^{-1}\left((I-\gamma P^{\pi_{k+1}})\theta_k - g^{\pi_{k+1}} \right)\\
&= (I-\gamma P^{\pi_{k+1}})^{-1}g^{\pi_{k+1}}\\
&= \theta^{\text{PI}}_{k+1}\,.
\end{aligned}
\end{equation*}
The quadratic local contraction follows from Corollary~\ref{corollary Newton local contraction}.~~\QED
\end{proof}
The theoretical results of Proposition~\ref{proposition PI-Newton} are corroborated by extensive empirical evidence that suggests that, in practice, PI leads to faster convergence in terms of number of iterations than VI~\cite{bertsekas2012DPOC, gargiani21}. 
Despite its simplicity, the consequences of Proposition~\ref{proposition PI-Newton} are far-reaching, especially in light of the results in Theorem~\ref{theorem local contraction}. We can develop novel DP methods in the spirit of semismooth Newton-type methods, where the elements in the B-differential are approximated with non-singular matrices that verify the kappa condition~\eqref{eq: kappa_condition}. Assumption~\ref{assumption on spurious greedy policies} allows to directly employ Proposition~\ref{proposition B-differential of PC functions} and could be further relaxed by considering only the iterates $\theta_k$ for $k\geq 0$. In addition, despite its technicality and limited intuitiveness, empirical evidence seems to suggests that it is realistic to assume that $\tilde{\Pi}^{S}_{\theta_k} = \emptyset$ for all $k\geq 0$.


By adopting the piecewise smooth Newton perspective (see Theorem 7.2.15 in~\cite{facchineipangvol22003}) we can recover similar results as in Proposition~\ref{proposition PI-Newton} without the need for Assumption~\ref{assumption on spurious greedy policies}. In particular, $J$ is selected in the larger set $\hat{\partial} r(\theta)\supseteq \partial_B r(\theta)$ that comprises the Jacobians of all the active selection functions at $\theta$. Clearly $\hat{\partial} r(\theta)$ also contains the Jacobians of the active selection functions associated with the spurious greedy policies.
With this approach, Assumption~\ref{assumption on spurious greedy policies} is replaced by the requirement that $\hat{\partial} r(\theta)$ is a \textit{strong Newton approximation scheme} (see Definition 7.2.2 in~\cite{facchineipangvol22003}). 

Also the analysis of Bertsekas in~\cite{bertsekas2022} leads to similar conclusions on the local convergence of PI. Unlike our analysis though, Bertsekas considers a neighborhood of the root where the active selection functions are a subset of those active at the root. This allows to remap the iterations to the Newton iterations applied to a system of differentiable equations that has the same fixed point. The downside of this approach is that the effective region of attraction is potentially much larger than the one considered for the technical proof.
A clear example is depicted in Figures~\ref{figure: pieces_BE} and~\ref{figure: region_attraction}.

\subsection{Value Iteration}
\label{sec:VI and Gradient}
In light of the equivalence between PI and the semismooth Newton method to solve~\eqref{eq: Bellman residual function}, we investigate the connection between VI and semismooth Newton-type methods. In particular, with the following proposition we show that VI is a semismooth Newton-type method where the elements in Clarke's generalized Jacobian are approximated with the identity matrix. 
\begin{proposition}\label{proposition VI and Newton-type}
VI is a semismooth Newton-type method to solve the Bellman residual function with $\left\{ B_k\right\} = \left\{ I\right\}$.
\end{proposition}
\begin{proof}
Let $\theta^{\text{VI}}_{k+1}$ and $\theta^{\text{N-type}}_{k+1}$ denote the $(k+1)$-th iterate of VI and the semismooth Newton-type method with $\left\{ B_k\right\} = \left\{ I\right\}$, respectively. Assume that $\theta^{\text{VI}}_{k} = \theta^{\text{N-type}}_{k} = \theta_{k}$. Then, from the definition of VI, it follows that
$
\theta_{k+1}^{\text{VI}} = T\theta_k\,.
$
From the definition of semismooth Newton-type iterate in~\eqref{eq: Newton-type iteration closed-form} and with the specific choice of $B_k = I$, we obtain that
$
\theta_{k+1}^{\text{N-type}} 
= \theta_k - I^{-1}r(\theta_k) = \theta_k - \left( \theta_k - T\theta_k \right)=  T\theta_k = \theta^{\text{VI}}_{k+1}\,.~~\QED
$
\end{proof}

The classical DP convergence analysis of VI based on the properties of the Bellman operator indicates that VI enjoys a global linear rate of convergence with a $\gamma$-contraction rate. In light of this novel connection between VI and the fixed-point iteration method, we can adopt the semismooth Newton-type theory perspective to study the local convergence of VI. In particular, from the results of Theorem~\ref{theorem local contraction}, we obtain that VI has a local linear contraction rate given by the discount factor as
$
\Vert I^{-1}\left( I - \left( I-\gamma P^{\pi} \right) \right) \Vert_{\infty} = \gamma \Vert  P^{\pi} \Vert_{\infty} = \gamma < 1$ for all $\pi\in\Pi\,.  
$

\subsection{$\alpha$-Value Iteration}
\label{sec: alpha VI}
Proposition~\ref{proposition VI and Newton-type} shows that VI is also an instance of the fixed-point iteration method with $\alpha_k=1$ for all $k$. The question that naturally arises is what do the iterates of the fixed-point iteration method correspond to if we allow $\alpha_k\neq 1$. In this spirit, we propose to use $\alpha I$ with $\alpha > 0$ to approximate the elements in Clarke's generalized Jacobian.
 
The following lemma characterizes the iterates of this method, which we call $\alpha$-Value Iteration ($\alpha$-VI).
\begin{lemma}
Consider the semismooth Newton-type iteration for the Bellman residual function with $B_k = \alpha I$ and $\alpha>0$. Then
\vspace{-0.2cm}
\begin{equation}\label{eq: alpha-VI iteration}
\theta_{k+1} = \frac{\alpha-1}{\alpha}\theta_k + \frac{1}{\alpha} T \theta_k\,.
\end{equation}
\end{lemma}
\vspace{0.1cm}
\begin{proof}
We start from the semismooth Newton-type iteration in~\eqref{eq: Newton-type iteration closed-form} and set $B_k = \alpha I$. The result trivially follows from the definition of the Bellman residual function as
$
\theta_{k+1} 
= \theta_k - \frac{1}{\alpha} (\theta_k - T\theta_k)= \frac{\alpha-1}{\alpha}\theta_k + \frac{1}{\alpha}T\theta_k\,.~~\QED
$
\end{proof}

Starting from Equation~\eqref{eq: alpha-VI iteration}, we can define the operator
$
T_{\alpha} = \frac{\alpha-1}{\alpha}\, I + \frac{1}{\alpha}\, T\,,
$
where $I$ is the indentity map and $T$ is the Bellman operator. Notice that when $\alpha=1$ we recover the Bellman operator and therefore $1$-VI is simply VI.
In the following, we are interested in studying the global and local convergence of $\alpha$-VI. We start by studying the properties of the $T_{\alpha}$ operator and its fixed-points.
\setlength{\belowcaptionskip}{-10pt}
\begin{figure}[h]
\centering
\includegraphics[width=0.45\textwidth]{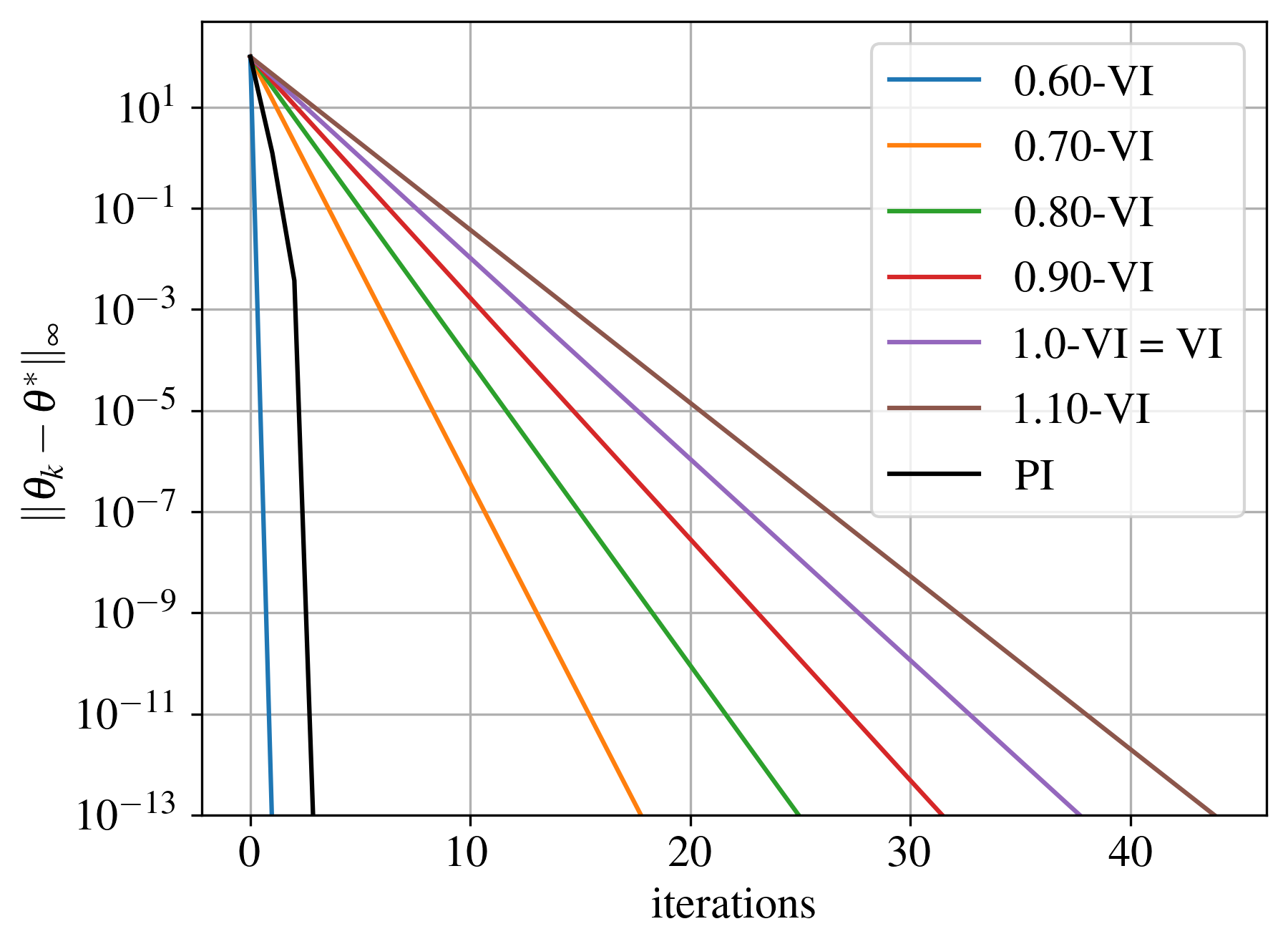}
\caption{Comparison of PI and $\alpha$-VI for different values of $\alpha$. For the benchmark, we consider a randomly generated MDP with $500$ states, $10$ actions and $\gamma=0.4$. In particular, the state transition matrices and the cost vectors are generated by sampling the values from a uniform distribution on the interval $[0,1)$.}
\label{figure: alpha_VI}
\end{figure}
\begin{proposition}\label{proposition contractivity of alpha-VI}
For any $\theta, \,\bar{\theta} \in \mathbb{R}^n$ and $\alpha > \frac{1+\gamma}{2}$,
$$ \big\Vert T_{\alpha}\theta - T_{\alpha}\bar{\theta} \big\Vert_{\infty} \leq \beta \big\Vert \theta - \bar{\theta} \big\Vert_{\infty}\,, $$
where $\beta = \frac{\vert\alpha-1\vert}{\alpha} + \frac{\gamma}{\alpha} <1$. In addition, the optimal cost $\theta^*$ is the unique fixed-point of $T_{\alpha}$.
\end{proposition}
\begin{proof}
We start by showing that, if $\alpha > \frac{1+\gamma}{2}$, the operator is $\beta$-contractive with respect to the infinity norm. For any $\theta,\,\bar{\theta}\in\mathbb{R}^n$
\begin{equation*}
\begin{aligned}
\Vert T_{\alpha}\theta - &T_{\alpha}\bar{\theta} \Vert_{\infty} \!= \max_{s\in\mathcal{S}} \Big\vert \frac{\alpha-1}{\alpha}\!\left(\theta_s -  \bar{\theta}_s\right) \!+\! \frac{1}{\alpha}\!\left(T\!\left(\theta - \bar{\theta} \right)\right)\!(s)\Big\vert\\
&\overset{(a)}{\leq}\!\Big\vert\! \frac{\alpha-1}{\alpha} \!\Big\vert \max_{s\in\mathcal{S}} \!\Big\vert \theta_s \!-\! \bar{\theta}_s \Big\vert \!  + \!   \frac{1}{\vert \alpha\vert}\! \max_{s\in\mathcal{S}}\Big\vert \!\left( T\!\left(\theta - \bar{\theta}  \right)\!\right)\!\!(s)  \Big\vert\\
&\overset{(b)}{\leq} \left( \Big\vert\frac{\alpha-1}{\alpha}\Big\vert + \frac{\gamma}{\vert\alpha\vert} \right) \max_{s\in\mathcal{S}} \Big\vert \theta_s - \bar{\theta}_s \Big\vert\\
&= \,\left( \Big\vert\frac{\alpha-1}{\alpha}\Big\vert + \frac{\gamma}{\vert\alpha\vert} \right) \big\Vert \theta - \bar{\theta}  \big\Vert_{\infty}\,, 
\end{aligned}
\end{equation*}
where $(a)$ follows from the triangle inequality and $(b)$ from the fact that the Bellman operator is $\gamma$-contractive in the inifinity norm. In order for $T_{\alpha}$ to be contractive, we need $\left( \Big\vert\frac{\alpha-1}{\alpha}\Big\vert + \frac{\gamma}{\vert\alpha\vert} \right) <1$. For $\alpha \geq 1$, since $\gamma\in (0,1)$, $T_{\alpha}$ is contractive with rate $(\alpha-1)/\alpha + \gamma/\alpha$. For $\alpha \in (0,1)$, $\Big\vert \frac{\alpha-1}{\alpha}  \Big\vert + \frac{\gamma}{\vert \alpha \vert}  = \frac{1-\alpha}{\alpha} + \frac{\gamma}{\alpha} $ and $\frac{1-\alpha}{\alpha} + \frac{\gamma}{\alpha} <1$ if and only if $\alpha > \frac{1+\gamma}{2}$.  For $\alpha <0$, $\Big\vert \frac{\alpha-1}{\alpha}  \Big\vert + \frac{\gamma}{\vert \alpha \vert}= \frac{\alpha-1}{\alpha} - \frac{\gamma}{\alpha}$ and the inequality $\frac{\alpha-1}{\alpha} - \frac{\gamma}{\alpha}<1$ is never satisfied since $\gamma\in (0,1)$.  We can therefore conclude that if $\alpha > \frac{1+\gamma}{2}$ then $T_{\alpha}$ is $\beta$-contractive in the infinity norm with $\beta = \vert \alpha - 1 \vert/\alpha + \gamma/\alpha$. To verify that $\theta^*$ is a fixed-point of $T_{\alpha}$, we exploit the definition of $T_{\alpha}$ and the fact that $\theta^*$ is the unique fixed-point of $T$. In particular,
$
T_{\alpha}\theta^* = \frac{\alpha-1}{\alpha}\theta^* + \frac{1}{\alpha}T\theta^*
= \frac{\alpha-1}{\alpha}\theta^* + \frac{1}{\alpha}\theta^*
= \theta^*\,.
$
Uniqueness follows directly from the Banach Theorem~\cite{rockafellar76}.~~\QED
\end{proof}
The main implication of Proposition~\ref{proposition contractivity of alpha-VI} is that, if $\alpha>(1+\gamma)/2$, then $\alpha$-VI converges globally to the optimal cost $\theta^*$ with linear rate $\beta$.
The following lemmas characterize the values of $\alpha$ for which $T_{\alpha}$ is a monotone operator and its shift-invariance property, respectively.
\begin{lemma}[monotonicity]
Let $\alpha \geq 1$. For $\theta, \,\bar{\theta}\in \mathbb{R}^n$ if $\theta \leq \bar{\theta}$, then  $T_{\alpha}\theta \leq T_{\alpha}\bar{\theta}$.
\end{lemma}
\begin{proof} Since $\alpha\geq 1$, $\theta\leq \bar{\theta}$ and $T$ is monotone~\cite{bertsekas2012DPOC}, it follows that
$
T_{\alpha}\theta = \frac{\alpha-1}{\alpha} \theta + \frac{1}{\alpha}T\theta\leq \frac{\alpha-1}{\alpha} \bar{\theta} + \frac{1}{\alpha}T\bar{\theta}= T_{\alpha}\bar{\theta}\,.~\QED
$
\end{proof}
\begin{lemma}[shift-invariance]
For any $\theta\in\mathbb{R}^n$ and $b\in\mathbb{R}$, then $T^{k}_{\alpha}\left(\theta + b \mathbf{1}_n \right) = T^k_{\alpha}\theta + \left(\frac{\alpha-1+\gamma}{\alpha}\right)^k b \mathbf{1}_n$ for $k=1,2,\dots$.
\end{lemma}
\begin{proof}
Since $T$ is shift-invariant~\cite{bertsekas2012DPOC}, then
\begin{equation*}
\begin{aligned}
T_{\alpha}\left( \theta + b \mathbf{1}_n \right) &=  \frac{\alpha-1}{\alpha}\left(\theta + b\mathbf{1}_n \right)+ \frac{1}{\alpha}T\left(\theta + b\mathbf{1}_n \right)\\
&= \frac{\alpha - 1}{\alpha}\left(\theta + b\mathbf{1}_n \right)+ \frac{1}{\alpha}T\theta + \frac{\gamma}{\alpha} b\mathbf{1}_n \\
&= T_{\alpha}\theta + \frac{\alpha-1+\gamma}{\alpha}b\mathbf{1}_n\,.
\end{aligned}
\end{equation*}
The final result follows from repeatedly applying the $T_{\alpha}$ operator.~~\QED 
\end{proof}
Results similar to Proposition~\ref{proposition contractivity of alpha-VI} can be derived for the local contraction rate by considering Theorem~\ref{theorem local contraction} and evaluating the kappa condition with the infinity norm.
Unfortunately, using this type of analysis it is not possible to conclude that $\alpha$-VI improves over VI in terms of convergence rate. Instead, we introduce the following proposition, which analyses the asymptotic rate of convergence of $\alpha$-VI via local stability analysis of nonlinear systems. For the sake of simplicity and interpretability, we consider a simplified setting in which the transition probability matrix at the solution has only real and positive eigenvalues. Notice that similar considerations can be made in a more general setting. This approach provides a tighter bound on the local rate of convergence, but is only applicable in a neighborhood of the root where the Bellman residual function is continuously differentiable.

\begin{proposition}[asymptotic local contraction rate]\label{proposition on asymptotic local contraction rate}
Assume that $r(\theta^*)$ is continuously differentiable in a neighborhood of $\theta^*$ and that $P^{\pi^*}$ has only real and positive eigenvalues. Let $\alpha \in (1/(1+\gamma), 1)$ and
\begin{equation}
\tilde{\beta} = 
\begin{cases}
1-\frac{1-\gamma}{\alpha}& \text{if }\alpha\in [1-\gamma/2, 1)\\
\frac{1}{\alpha}-1 & \text{if }\alpha \in (1/(1+\gamma), 1-\gamma/2)\,.
\end{cases}
\end{equation}
$\alpha$-VI converges linearly to $\theta^*$ with asymptotic contraction rate $\tilde{\beta}<\gamma$. 
\end{proposition}
\begin{proof}
We start by linearizing $\theta_{k+1}= T_{\alpha}\theta_k$ at $\theta^*$ via the first-order Taylor expansion
\begin{equation*}
\theta^*\! + I (\theta_{k+1} - \theta^*) \!=\! T_{\alpha}\theta^* + (T_{\alpha}\theta^*)^{\prime} (\theta_k - \theta^*) + \mathcal{O}\left(\Vert \theta_k - \theta^* \Vert^2 \right).
\end{equation*}
Since $\theta^* = T_{\alpha}\theta^*$ and $(T_{\alpha}\theta^*)^\prime = \frac{(\alpha-1)}{\alpha} I + \frac{\gamma}{\alpha} P^{\pi^*}$ for any optimal policy $\pi^*$, then
\begin{equation*}
\theta_{k+1} - \theta^*  \!=\! \left(\! I \!- \!\frac{1}{\alpha}\left( I - \gamma P^{\pi^*}\right)\!\right) \!(\theta_k - \theta^*) + \mathcal{O}\!\left( \Vert \theta_k - \theta^* \Vert^2\right). 
\end{equation*}
Therefore the asymptotic convergence rate is determined by the spectral radius of $I - \frac{1}{\alpha}\left( I - \gamma P^{\pi^*}\right)$. In particular, since $\rho\left(I - \frac{1}{\alpha}\left( I - \gamma P^{\pi^*}\right) \right) \leq \max\left\{\big\vert 1-\frac{1-\gamma}{\alpha} \big\vert, \big\vert 1-\frac{1}{\alpha} \big\vert  \right\}$, we study different cases based on the values of $\alpha$. When $\alpha\geq 1-\gamma/2$, then $\max\left\{\big\vert 1-\frac{1-\gamma}{\alpha} \big\vert, \big\vert 1-\frac{1}{\alpha} \big\vert  \right\} = 1-\frac{1-\gamma}{\alpha}$. In this case we get a contraction for any $\alpha\geq 1-\gamma/2$ since the inequality $1-\frac{1-\gamma}{\alpha}<1$ is verified for any $\alpha>0$. In addition, if $\alpha \in [1-\gamma/2, 1]$, then we improve over the rate of VI since $1 - \frac{1-\gamma}{\alpha} \leq \gamma$. For $\alpha < 1-\gamma/2$, 
$\max\left\{ \big\vert 1-\frac{1-\gamma}{\alpha} \big\vert, \big\vert 1-\frac{1}{\alpha} \big\vert \right\} = \frac{1}{\alpha}-1$ and we get a contraction if $\alpha\in(1/2, 1-\gamma/2)$. In addition, if $\alpha\in [{1}/({1+\gamma}), 1-\gamma/2)$, then $\frac{1}{\alpha}-1 \leq \gamma$ and therefore we improve over the rate of VI.~~\QED
\end{proof}
By combining the results of Propositions~\ref{proposition contractivity of alpha-VI} and~\ref{proposition on asymptotic local contraction rate} we obtain that, by setting $\max\left\{ \frac{1}{1+\gamma}, \frac{1+\gamma}{2} \right\}<\alpha<1$, $\alpha$-VI converges globally with a linear rate and its asymptotic linear rate of convergence is strictly better than that of VI. The numerical experiments in Figures~\ref{figure: alpha_VI} and \ref{figure: alpha_graph} corroborate our theoretical findings and demonstrate the competitive performance of $\alpha$-VI. In addition, since our analysis is not tight, in practice we obtain convergence for a wider range of $\alpha$ as depicted in Figure~\ref{figure: alpha_graph}. The code is available at \url{https://gitlab.ethz.ch/gmatilde/alphaVI}. 
\setlength{\belowcaptionskip}{-10pt}
\begin{figure}[h]
\centering
\includegraphics[width=0.45\textwidth]{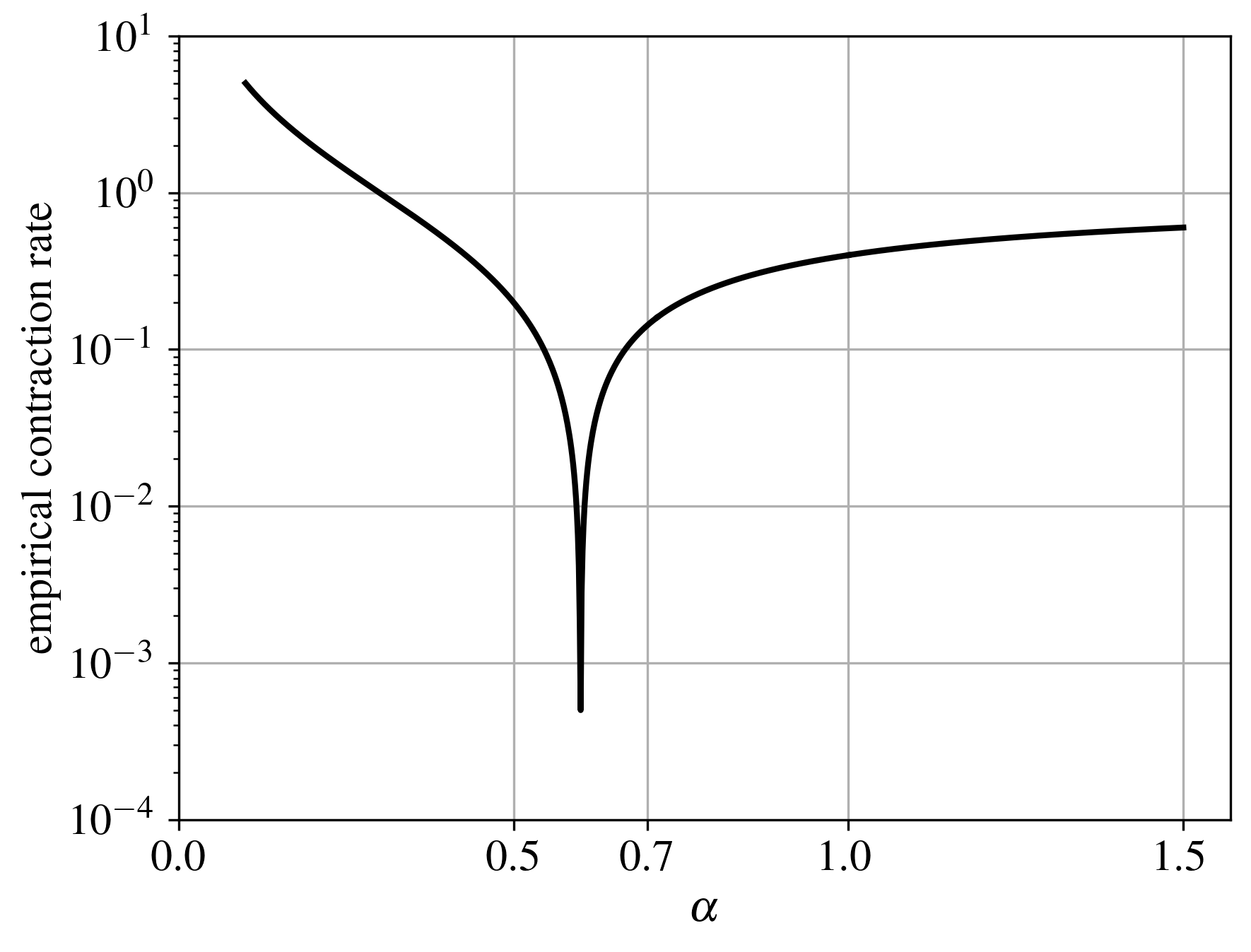}
\caption{Empirical global contraction rate of $\alpha$-VI for different values of $\alpha$ and comparison of $\alpha$-VI and PI for a randomly generated MDP with 500 states, 10 actions and $\gamma$ = 0.4.  The maximum acceleration is quite dramatic and is obtained for $\alpha\approx 0.6$.}
\label{figure: alpha_graph}
\end{figure}

\section{CONCLUSIONS \& FUTURE WORK}
We developed a unified convergence analysis for semismooth Newton-type methods based on the kappa condition. We then proved that PI and VI are semismooth Newton-type methods. 
In particular, Propositions~\ref{proposition PI-Newton} and~\ref{proposition VI and Newton-type} reveal that PI and VI sit at the two opposite sides in the spectrum of semismooth Newton-type methods: PI enjoys local quadratic contraction but its costs per iteration are demanding; instead, VI is based on a coarse approximation of the elements in Clarke's generalized Jacobian which allows to drastically reduce the costs per iteration at the price of downgrading the local quadratic convergence to a linear one. 
This connection has far-reaching consequences on the theoretical and algorithmic side.
We can both deploy the semismooth Newton-type theory to analyze the local convergence properties of existing DP methods and, taking inspiration from the existing semismooth Newton-type methods, design novel DP algorithms that achieve different trade-offs of local contraction rate and costs per iteration. 
In this spirit, we proposed an extension of VI with global convergence guarantees and asymptotically faster contraction rate. This novel locally accelerated version of VI comes with negligible additional computational costs and leads to great improvement in performance, as demonstrated by our numerical experiments.

Finally, another promising future direction consists in formalizing and exploiting the connection between inexact semismooth Newton methods and optimistic policy iteration-type algorithms.


\bibliographystyle{plain} 
\bibliography{refs}

\begin{thebibliography}{10}

\bibitem{bertsekas2012DPOC}
D.~P. Bertsekas.
\newblock {\em Dynamic Programming and Optimal Control}.
\newblock Athena Scientific, 2012.

\bibitem{bertsekas2022}
D.~P. Bertsekas.
\newblock {\em Lessons from AlphaZero for Optimal, Model Predictive, and
  Adaptive Control}.
\newblock Athena Scientific, 2022.
\newblock Forthcoming Book.

\bibitem{boyd2004convex}
S.~Boyd and L.~Vandenberghe.
\newblock {\em Convex Optimization}.
\newblock Cambridge University Press, 2004.

\bibitem{clarke1990optimization}
F.~H. Clarke.
\newblock {\em Optimization and Nonsmooth Analysis}.
\newblock SIAM, 1990.

\bibitem{diehl2016}
M.~Diehl.
\newblock Lecture notes on numerical optimization.
\newblock Leuven-Freiburg 2007-2015 (last update: 02.02.2016).

\bibitem{facchineipangvol22003}
F.~Facchinei and J.~Pang.
\newblock {\em Finite-Dimensional Variational Inequalities and Complementarity
  Problems}, volume~2.
\newblock Springer, 2003.

\bibitem{facchineipangvol12003}
F.~Facchinei and J.~Pang.
\newblock {\em Finite-Dimensional Variational Inequalities and Complementarity
  Problems}, volume~1.
\newblock Springer, 2003.

\bibitem{gargiani21}
M.~Gargiani, A.~Martinelli, M.~Martinez, and J.~Lygeros.
\newblock Parallel and flexible dynamic programming via the randomized
  mini-batch operator.
\newblock arXiv:2110.02901, 2021.

\bibitem{izmasolo14}
A.~Izmailov and M.~Solodov.
\newblock {\em Newton-Type Methods for Optimization and Variational Problems}.
\newblock Springer, 2014.

\bibitem{10.1145/2491491.2491493}
K.~A. Khan and P.~I. Barton.
\newblock Evaluating an element of the {C}larke generalized jacobian of a
  composite piecewise differentiable function.
\newblock {\em ACM Trans. Math. Softw.}, 39(4), 2013.

\bibitem{kummer88}
B.~Kummer.
\newblock Newton's method for non-differentiable functions.
\newblock {\em Advances in Math. Optimization.}, 45:114--125, 12 1988.

\bibitem{pollatschek1969}
M.~Pollatschek and B.~Avi-Itzhak.
\newblock Algorithms for stochastic games with geometrical interpretation.
\newblock {\em Management Science}, 15:399--413, 1969.

\bibitem{puterman1979}
M.~L. Puterman and S.~L. Brumelle.
\newblock On the convergence of policy iteration in stationary dynamic
  programming.
\newblock {\em Mathematics of Operations Research}, 4(1):60--69, 1979.

\bibitem{qi1993nonsmooth}
L.~Qi and J.~Sun.
\newblock A nonsmooth version of {N}ewton's method.
\newblock {\em Mathematical programming}, 58(1):353--367, 1993.

\bibitem{rademacher1919partielle}
H.~Rademacher.
\newblock {\"U}ber partielle und totale {D}ifferenzierbarkeit von {F}unktionen
  mehrerer {V}ariabeln und {\"u}ber die {T}ransformation der {D}oppelintegrale.
\newblock {\em Mathematische Annalen}, 79(4):340--359, 1919.

\bibitem{rockafellar76}
R.~T. Rockafellar.
\newblock Monotone operators and the proximal point algorithm.
\newblock {\em Mathematics of Operations Research}, 14(5):877--898, 1996.

\bibitem{santos2004}
M.~S. Santos and J.~Rust.
\newblock Convergence properties of policy iteration.
\newblock {\em SIAM J. on Control and Optimization}, 42:2094--2115, 2004.

\bibitem{alphago2016}
D.~Silver, A.~Huang, C.~J. Maddison, A.~Guez, L.~Sifre, G.~van~den Driessche,
  J.~Schrittwieser, I.~Antonoglou, V.~Panneershelvam, M.~Lanctot, S.~Dieleman,
  D.~Grewe, J.~Nham, N.~Kalchbrenner, I.~Sutskever, T.~Lillicrap, M.~Leach,
  K.~Kavukcuoglu, T.~Graepel, and D.~Hassabis.
\newblock Mastering the game of go with deep neural networks and tree search.
\newblock {\em Nature}, 529:484--503, 2016.

\end{thebibliography}

\end{document}